\documentclass{amsart} 

\usepackage{amsthm, amssymb, amsmath, amscd}
\usepackage{eurosym}
\usepackage{amsfonts}
\usepackage{amsmath}
\usepackage{setspace}
\usepackage{graphicx}%

\usepackage[all,cmtip]{xy}

\usepackage{indentfirst}

\newtheorem{theorem}{Theorem}[section]

\newtheorem{cor}[theorem]{Corollary}

\newtheorem{lemma}[theorem]{Lemma}

\newtheorem{prop}[theorem]{Proposition}

\theoremstyle{definition}
\newtheorem{define}[theorem]{Definition}
\newtheorem{ex}[theorem]{Example}
\newtheorem{remark}[theorem]{Remark}

\newcommand{\field}[1]{\mathbb{#1}}

\newcommand{\zkz}{\field{Z}/k\field{Z}}
\def\revdots{\mathinner{\mkern1mu\raise1pt\vbox{\k ern7pt\hbox{.}}\mkern2mu\raise4pt\hbox{.}\mkern2mu \raise7pt\hbox{.}\mkern1mu}}

\begin{document}
\title[$K$-homology with coefficients II]{Geometric $K$-homology with coefficients II: \\ The Analytic Theory and Isomorphism}
\author{Robin J. Deeley}
\thanks{}
\subjclass[2010]{Primary: 19K33; Secondary: 19K56, 46L85, 55N20  }
\keywords{$K$-homology, $\zkz$-manifolds, the Freed-Melrose index theorem}
\date{}
\begin{abstract}
We discuss the analytic aspects of the geometric model for $K$-homology with coefficients in $\zkz$ constructed in \cite{Dee1}.  In particular, using results of Rosenberg and Schochet, we construct a map from this geometric model to its analytic counterpart.  Moreover, we show that this map is an isomorphism in the case of a finite CW-complex.  The relationship between this map and the Freed-Melrose index theorem is also discussed.  Many of these results are analogous to those of Baum and Douglas in the case of ${\rm spin^c}$ manifolds, geometric K-homology, and Atiyah-Singer index theorem.
\end{abstract}
\maketitle

\section{Introduction}
This is the second in a pair of papers whose topic is the construction of a geometric model for $K$-homology with coefficients using ${\rm spin^c}$ $\zkz$-manifold theory.  Despite this, we have tried to make our treatment as self-contained as possible.  In the first paper, \cite{Dee1}, the cycles and relations for this model were described.  In addition, it was shown that the model fits into the correct Bockstein sequence for the coefficient group $\zkz$.  The goal of this paper is the construction (based on results in \cite{Ros1} and \cite{Sch}) of an analytic model for $K$-homology with coefficients in $\zkz$ and the construction of a map (defined at the level of cycles) from the geometric model in \cite{Dee1} to this analytic model.  The main result is that this map is an isomorphism for finite CW-complexes (see Theorem \ref{thmIsoGeoAnazkz}).  \par
The reader should note the similarity with a number of constructions due to Baum and Douglas (see \cite{BD, BDbor}).  However, a number of our constructions involve noncommutative $C^*$-algebras (rather than the $C^*$-algebra of continuous function on a manifold in the case considered by Baum and Douglas).  These algebras are constructed (based on \cite{Ros1}) using groupoids, but we have endeavoured to make them accessible to the reader unfamiliar with groupoid $C^*$-algebras.  We also use some KK-theory but the amount is quite limited.  Thus, prerequisites are limited to an understanding of the Baum-Douglas model for (geometric) K-homology and the Fredholm module picture of (analytic) K-homology (due to Kasparov, \cite{Kas}).  
\par
The first section summarizes results contained in \cite{Dee1}.  The reader is directed to \cite{BHS} for more on geometric K-homology, \cite{HR} for more on analytic K-homology and \cite{Dee1} (and the references therein) for more on $\zkz$-manifolds, the Freed-Melrose index theorem and the construction of geometric K-homology with coefficients in $\zkz$.  We note that $\zkz$-manifolds were first introduced by Sullivan (see \cite{MS, Sul1, Sul}).  \par
The second section contains the main results of the paper.  Namely, the construction of the map from geometric K-homology with coefficients in $\zkz$ to analytic K-homology with coefficients in $\zkz$ and the proof that this map is an isomorphism in the case of a finite CW-complex.  To put these results in context, we review the analogous results in K-homology (c.f., \cite{BHS}).  The reader should recall that a geometric cycle in K-homology is given by a triple, $(M,E,f)$, where $M$ is compact ${\rm spin^c}$ manifold, $E$ is a vector bundle, and $f$ is a continuous map from $M$ to $X$; ($X$ is the space whose K-homology we are modeling).  An analytic cycle is given by a Fredholm module over $C(X)$.  The map from the first of these theories to the latter is defined in three steps:
\begin{enumerate}
\item To the ${\rm spin^c}$ manifold in a geometric cycle, $(M,E,f)$, we associated a $C^*$-algebra, $C(M)$;
\item We use the ${\rm spin^c}$-structure and the vector bundle, $E$, to produce a Fredholm module (denoted $[D_E]$) in the K-homology of $C(M)$;
\item Finally, the continuous map induces a map from the K-homology of $C(M)$ to the K-homology of $C(X)$ (denoted by $f_*$), which we apply to $[D_E]$ to get a class in the K-homology of $C(X)$.
\end{enumerate}    
The map from geometric cycles to analytic cycles is then defined to be 
\begin{equation}
\mu: (M,E,f) \mapsto f_*([D_E]) \label{isoNoCoeff}
\end{equation}
In \cite{BHS}, this map is shown to be an isomorphism in the case when $X$ is a finite CW-complex. 
\par
The results in the case of $\zkz$-coefficients can be summarized as follows.  We recall (see Definitions \ref{zkzmfld} and \ref{zkzcyc} below) that a geometric $\zkz$-cycle is a triple, $((Q,P),(E,F),f)$, where $(Q,P)$ is a compact ${\rm spin^c}$ $\zkz$-manifold, $(E,F)$ is a $\zkz$-vector bundle, and $f$ is a continuous map from $(Q,P)$ to $X$ (the space whose K-homology with coefficients in $\zkz$ we are modeling).  On the other hand, an analytic $\zkz$-cycle is a Fredholm module over the $C^*$-algebra $C(X)\otimes C^*(pt;\zkz)$, where $C^*(pt;\zkz)$ is the mapping cone of the inclusion of $\field{C}$ into the $k$ by $k$ matrices.  (Based on results in \cite{Sch}, we define analytic K-homology with coefficients in $\zkz$ to be $K^*(C(X)\otimes C^*(pt;\zkz))$.)  The map (which is analogous to the map $\mu$ above) from geometric $\zkz$-cycles to analytic $\zkz$-cycles is defined in three steps:
\begin{enumerate}
\item Using Rosenberg's construction (see \cite{Ros1}), we associated a $C^*$-algebra to the $\zkz$-manifold, $(Q,P)$ (denoted by $C^*(Q,P;\zkz)$);
\item Again, following Rosenberg, we use the ${\rm spin^c}$-structure and the $\zkz$-vector bundle, $(E,F)$, to produce a Fredholm module (denoted $[D_{(E,F)}]$) in the K-homology of $C(Q,P;\zkz)$;
\item Finally, the continuous map, $f$, induces a $*$-homomorphism from $C(X)\otimes C^*(pt;\zkz)$ to $C^*(Q,P;\zkz)$ and hence a map from the K-homology of $C^*(Q,P;\zkz)$ to the K-homology of $C(X)\otimes C^*(pt;\zkz)$.  We denote this map by $\tilde{f}^*$ and then apply it to $[D_{(E,F)}]$ to get a class in the K-homology of $C(X)\otimes C^*(pt;\zkz)$.
\end{enumerate}
To summarize, the map from the geometric theory to the analytic theory is defined via
$$\Phi: ((Q,P),(E,F),f) \mapsto \tilde{f}^*([D_{(E,F)}])$$  
The proof that this map is well-defined in not trivial.  The most involved part is the bordism relation (see Theorem \ref{coborInvClass}).  To show that $\Phi$ is an isomorphism for finite CW-complexes, we use the Bockstein sequences for both the geometric and analytic models, the Five Lemma, and the fact that $\mu$ is an isomorphism for finite CW-complexes (see \cite[Theorem 6.2]{BHS}). 
\par
In the final section of the paper, we discuss the relationship between our results and index theory.  In particular, we discuss the fact that it follows from our construction that the Freed-Melrose index theorem for ${\rm spin^c}$ $\zkz$-manifolds can be conceptualized as a specific case of the isomorphism from geometric to analytic $K$-homology with coefficient in $\zkz$.  This is analogous to Baum and Douglas' conceptualization of the Atiyah-Singer index theorem as a specific case of the isomorphism between geometric and analytic $K$-homology.  We also discuss index pairings in this section.

\subsection{Geometric $\zkz$-cycles} \label{zkzSec}
We now discuss the main results of \cite{Dee1}. 
\begin{define} \label{zkzmfld}
Let $Q$ be an oriented, smooth compact manifold with boundary.  We assume that the boundary of $Q$, $\partial Q$, decomposes into $k$ disjoint manifolds, $(\partial Q)_1, \ldots, (\partial Q)_k$.  A $\zkz$-structure on $Q$ is an oriented manifold $P$, a disjoint collaring neighbourhood, $V_i$ for each $(\partial Q)_i$, and orientation preserving diffeomorphisms $\gamma_i: V_i \rightarrow (0,1]\times P$.  A $\zkz$-manifold is a $Q$ with fixed $\zkz$-structure.  We denote this by $(Q,P,\gamma_i)$.  We sometimes drop the maps from this notation and denote a $\zkz$-manifold by $(Q,P)$. 
\end{define}
Many concepts from differential geometry and topology have natural generalizations from the manifold setting to the $\zkz$-manifold setting.  The generalization of vector bundles to $\zkz$-vector bundles is prototypical.  A $\zkz$-vector bundle is a pair, $(E,F)$, where $E$ is a vector bundle over $Q$, $F$ is a vector bundle over $P$, and $E|_{\partial Q}$ decomposes into k copies of $F$.  To be more precise, the identification of (i.e., isomorphism between) $E|_{\partial Q}$ and the k-copies of $F$ is also considered part of the data.  Additionally, we have natural definitions of a $\zkz$-Riemannian metric, a $\zkz$-fiber bundle, a ${\rm spin^c}$-structure on a $\zkz$-vector bundle, and a ${\rm spin^c}$-structure on a $\zkz$-manifold. The reader can see \cite[Definition 3.1]{Hig} for further details.
\begin{ex}
We consider the manifold with boundary, denoted by $Q$, given in Figure \ref{z3pic} and take $P=S^1$.  Then one can easily see that $(Q,P)$ has the structure of a $\field{Z}/3$-manifold. \label{z3ex}
\end{ex}
\begin{define}
Let $\bar{Q}$ be an $n$-dimensional, oriented, smooth, compact manifold with boundary.  In addition, assume we are given $k$ disjoint, oriented embeddings of an $(n-1)$-dimensional, oriented, smooth, compact manifold with boundary, $\bar{P}$, into $\partial \bar{Q}$.  Using the same notation as Definition \ref{zkzmfld}, we denote this as a triple, $(\bar{Q},\bar{P},\gamma_i)$ (or just $(\bar{Q},\bar{P})$), where $\{\gamma_i\}_{i=1}^k$ denote the $k$ disjoint oriented embeddings.  Such a triple is called a $\zkz$-manifold with boundary.  Its boundary is defined to be the $\zkz$-manifold, $(\partial \bar{Q} - {\rm int}(k\bar{P}), \partial \bar{P})$, where $k\bar{P}$ denotes the $k$ copies of $\bar{P}$ in $\partial \bar{Q}$.  In particular, if a $\zkz-$manifold $(Q,P)$ is the boundary of the $\zkz$-manifold with boundary, $(\bar{Q},\bar{P})$, then $\partial \bar{Q}  =  Q \cup_{\partial Q} ( k\bar{P} )$ and $\partial \bar{P}  =  P$.
\label{zkzmfldbound}
\end{define}
\begin{define}
Let $X$ be a compact space.  A $\zkz$-cycle over $X$ is a triple, $((Q,P),(E,F),f)$, where $(Q,P)$ is a ${\rm spin^c}$ $\zkz$-manifold, $(E,F)$ is a $\zkz$-vector bundle and $f$ is a continuous map from $(Q,P)$ to $X$. \label{zkzcyc}
\end{define}
The reader should note that the continuous map from $(E,F)$ to $X$ must respect the $\zkz$-structure.  If the compact space ($X$ in Definition \ref{zkzcyc}) is clear from the context, then we will refer to $\zkz$-cycles, rather than $\zkz$-cycles over $X$.   
\newpage
\begin{figure}
    \centering
        \includegraphics[width=8cm, height=5cm]{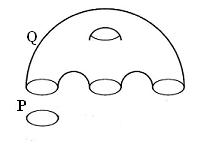} 
    \caption{$\field{Z}/3$-manifold from Example \ref{z3ex}.}
    \label{z3pic}
\end{figure}
 As with the Baum-Douglas model, the addition operation is defined using disjoint union and the inverse of a cycle is given by taking its ``opposite" (see \cite{Dee1}).  
\begin{define} \label{zkzCycleBordism}
A $\zkz$-cycle, $((Q,P),(E,F),f)$, is a boundary if there exist 
\begin{enumerate}
\item A smooth compact ${\rm spin^c}$ $\zkz$-manifold with boundary, $(\bar{Q},\bar{P})$, 
\item A smooth Hermitian $\zkz$-vector bundle $(V,W)$ over $(\bar{Q},\bar{P})$, 
\item A continuous map $\Phi :(\bar{Q},\bar{P}) \rightarrow X$, 
\end{enumerate}
such that $(Q,P)$ is the $\zkz$-boundary of $(\bar{Q},\bar{P})$, $(E,F)= (V,W)|_{\partial (\bar{Q},\bar{P})}$, and $f=\Phi|_{\partial (\bar{Q},\bar{P})}$. We say that $((Q,P),(E,F),f)$ is bordant to $((\hat{Q},\hat{P}),(\hat{E},\hat{F}),\hat{f})$ if 
\begin{equation*}
((Q,P),(E,F),f)\dot{\cup} \;(-(\hat{Q},\hat{P}),(\hat{E},\hat{F}),\hat{f})
\end{equation*}
is a boundary.
\end{define}
\begin{define} \label{vbmzkz}
Vector bundle modification for $\zkz$-cycles is defined as follows.  Let $((Q,P),(E,F),f)$ be a $\zkz$-cycle and $(W,V)$ be an even-dimensional ${\rm spin^c}$ $\zkz$-vector bundle over $(Q,P)$.  We note that $(Q,E,f)$ is a Baum-Douglas cycle with boundary and $(P,F,f|_P)$ is a Baum-Douglas cycle.  As such we can define the $\zkz$-vector bundle modification of $((Q,P),(E,F),f)$ by $(W,V)$ to be the Baum-Douglas vector bundle modification of the cycles $(Q,E,f)$ and $(P,F,f|_P)$ by $W$ and $V$ respectively.  The compatibility required by the definition of a $\zkz$-vector bundle ensures that the result of such a modification forms a $\zkz$-cycle. 
\end{define} 
\begin{define}
We define $K_*(X;\zkz)$ to be the set of equivalence classes of $\zkz$-cycles where the equivalence relation is generated by:  
\begin{enumerate}
\item If $\epsilon_1=((Q,P),(E_1,F_1),\phi)$ and $\epsilon_2=((Q,P),(E_2,F_2),\phi)$ are $\zkz$-cycles, then $$\epsilon_1 \dot{\cup} \epsilon_2 \sim ((Q,P),(E_1\oplus E_2,F_1\oplus F_2), \phi)$$
\item Bordant $\zkz$-cycles are defined to be equivalent; 
\item A $\zkz$-cycle is defined to be equivalent to its vector bundle modification by any even-dimensional ${\rm spin^c}$ vector bundle.
\end{enumerate}
\end{define}
The set $K_*(X;\zkz)$ is a graded abelian group with the operation of disjoint union.  The Bockstein sequence for the model takes the following form (see \cite[Theorem 2.20]{Dee1}).
\begin{theorem} \label{Bockstein}
If $X$ is finite CW-complex, then the following sequence is exact. \newline
\begin{center}
$\begin{CD}
K_0(X) @>k>> K_0(X) @>r>> K_0(X;\zkz) \\
@AA\delta A @. @VV\delta V \\
K_1(X;\zkz) @<r<<  K_1(X) @<k<< K_1(X) 
\end{CD}$
\end{center}
where the maps are 
\begin{enumerate}
\item $k : K_*(X) \rightarrow K_*(X)$ is given by multiplication by $k$;
\item $r: K_*(X) \rightarrow K_*(X;\zkz)$ takes a cycle $(M,E,f)$ to $((M,\emptyset),(E, \emptyset),f)$;  
\item $\delta : K_*(X;\zkz) \rightarrow K_{*+1}(X)$ maps the cycle $((Q,P),(E,F),f)$ to $(P,F,f)$.  
\end{enumerate}
\end{theorem}

\section{Main Results} \label{anaVsTopZkz} 
In this section, we deal with the isomorphism between geometric and analytic $K$-homology with coefficients in $\zkz$.  In \cite{Sch}, Schochet defines an analytic model for $K$-homology with coefficients.  We use results of Rosenberg (see \cite{Ros1}) to link Schochet's analytic cycles to $\zkz$-manifold theory.  This leads to the construction of a map from the geometric model developed in \cite{Dee1} (also see Section \ref{zkzSec} above) to this analytic model.  
\par
To do so, we introduce (and generalize) the construction of a groupoid $C^*$-algebra from a $\zkz$-manifold developed in \cite{Ros1}.  Then the map from geometric cycles to analytic cycles is defined and it is proved (under the condition that $X$ is a finite CW-complex) that it is an isomorphism.  The construction of $C^*$-algebras from $\zkz$-manifolds introduced in \cite{Ros1} uses the theory of groupoid $C^*$-algebras.  The theory of groupoid $C^*$-algebras is developed in great detail in \cite{Ren}.  We will not need the full power of this theory and the reader unfamiliar with it could possibly take Equation \ref{cStarZkz} in Example \ref{RosGroEx} as the definition of the $C^*$-algebra associated to a $\zkz$-manifold.

\subsection{Analytic $K$-homology with coefficients in $\zkz$}
\begin{define} \label{mappingConePt}
Let $C^*(pt;\zkz)$ denote the mapping cone of the inclusion of $\field{C}$ into the $k$ by $k$ matrices, $M_k$.  That is, we let
\begin{equation}
C^*(pt;\zkz):=\{f\in C_0([0,\infty),M_k)|f(0) \hbox{ is a multiple of }I_k\} \label{mapConek}
\end{equation}
\end{define}
Basic properties of mapping cones imply that 
$$0 \rightarrow C_0((0,\infty),M_k) \rightarrow C^*(pt;\zkz) \rightarrow \field{C} \rightarrow 0$$
is exact and that $K^0(C^*(pt;\zkz))  \cong  \zkz$ and $K^1(C^*(pt;\zkz))  \cong  0$.
Using \cite[Section 5-6]{Sch}, we then have the following definition. 
\begin{define}
Let $X$ be a compact Hausdorff space.  Then, $K^{ana}_p(X;\zkz):=K^{-p}(C(X)\otimes C^*(pt,\zkz))$. \label{anaKhomCoeffDef}
\end{define}
\begin{remark}
The Bockstein sequence for $K^{ana}_*(X;\zkz)$ is given by the six-term exact sequence associated to short exact sequence of $C^*$-algebras: 
$$0 \rightarrow C(X)\otimes C_0((0,\infty),M_k) \rightarrow C(X)\otimes C^*(pt;\zkz) \rightarrow C(X) \rightarrow 0 $$
\end{remark}
\newpage 
\begin{figure}
    \centering
        \includegraphics[width=8cm, height=7cm]{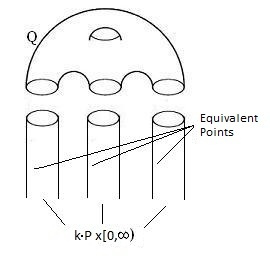} 
    \caption{The equivalence relation on the $\field{Z}/3$-manifold from Example \ref{z3ex}.}
    \label{groupoidC*zkz}
\end{figure}
\subsection{Rosenberg's Groupoid $C^*$-algebra}
We generalize Rosenberg's construction in \cite{Ros1} (also see \cite{Ros2}).  Rosenberg's construction is discussed in Example \ref{RosGroEx} below.  Moreover, Figure \ref{groupoidC*zkz} should be helpful to the reader during both the discussion of this example and the general construction.  \par 
We will work in a general framework, but Examples \ref{RosGroEx} and \ref{GroZkzWithBound} are our main concern.  The setup is the following.  Let $N$ be a smooth manifold.  The reader should note that $N$ may have boundary and is not necessarily compact.  Moreover, suppose that 
\begin{enumerate}
\item $N=M^Q \cup_{\Sigma} M^P$ where $M^Q$ and $M^P$ are manifolds with boundary and $\Sigma$ is a manifold (possibly with boundary);
\item $\Sigma \subseteq \partial M^Q$ and $\Sigma \subseteq \partial M^P$;
\item $M^P = k \cdot R$ for some manifold $R$ and $\Sigma = k \cdot \Sigma_R$ for some manifold $\Sigma_R$;
\end{enumerate}

An important case of this setup was considered by Rosenberg in \cite{Ros1}.  Let $(Q,P)$ be a $\zkz$-manifold and (using the notation above) let
\begin{equation}
M^Q=Q,\: M^P=\partial Q \times [0, \infty),\: \Sigma=\partial Q,\: R=P\times [0,\infty),\: \Sigma_R=P \label{spCas}
\end{equation}
We will discuss this case in more detail in Example \ref{RosGroEx} below. \par
Returning to the general setup, we construct a groupoid via an equivalence relation on $N$.  Figure \ref{groupoidC*zkz} illustrates the important special case (i.e., Equation \ref{spCas}).  The relation is defined by 
\begin{enumerate}
\item If $n \in M^Q$, then $n$ is equivalent only to itself.  We note that this includes points in $\partial Q$ and hence points in $\Sigma$.
\item If $n, n^{\prime} \in M^P-\Sigma$, then $n \sim n^{\prime}$ if and only if 
$$\pi(n)=\pi(n^{\prime}) \in R$$
where $\pi$ denotes the trivial covering map $M^P \rightarrow R$. 
\end{enumerate}
\begin{define} \label{groupoidC*algebraGen}
Using the notation and constructions in the previous paragraphs, we let $\mathcal{G} \subset N \times N$ denote the groupoid associated to this equivalence relation and let $C^*(\mathcal{G})$ denote the associated groupoid $C^*$-algebra.
\end{define}
A remark about notation is in order.  In the setup above, no assumption on compactness has been made; hence we will work with $C_0$-functions.  In particular, if $W$ is a manifold with boundary then $C_0(W)$ denotes continuous functions which vanish at $\infty$ but which take (possibly) nonzero values on the boundary of $W$.  While $C_0({\rm int}(W))$ denotes the continuous functions that vanish at $\infty$ and on the boundary.  We let $M_k$ denote the $k$ by $k$ matrices and if $M$ is a space, then $C_0(M,M_k)$ denotes the continuous functions from $M$ to $M_k$ which vanish at $\infty$.
\begin{prop}
Let $C^*(\mathcal{G})$ be the $C^*$-algebra from Definition \ref{groupoidC*algebraGen}.  Then it is isomorphic to 
\begin{equation*}
\{(f,g)\in C_0(M^Q) \oplus C_0(R,M_k) \mid g|_{\Sigma_R}\hbox{ is diagonal and }f|_{\Sigma}=g|_{\Sigma_R} \}
\end{equation*}
We note that the statement $f|_{\Sigma}=g|_{\Sigma_R}$ is more correctly written as 
$$\alpha(f|_{\Sigma})=g|_{\Sigma_R}$$
where $\alpha: C_0(\Sigma) \cong \oplus_{i=1}^k C_0(\Sigma_R) \rightarrow M_k(C_0(\Sigma_R))$ is the diagonal inclusion.
\label{PropAboutGroupoid}
\end{prop}

\begin{proof}
We only sketch the ideas of the proof, leaving the details for the interested reader.  To begin, we review some notation.  Recall (see the paragraphs preceding Definition \ref{groupoidC*algebraGen}) that $N=M^Q \cup_{\Sigma} M^P$ where $M^P=k\cdot R$ and $\mathcal{G}$ denotes the equivalence relation (defined above) on $N$.  We will denote an element of $\mathcal{G}$, $n_1 \sim n_2$, as $(n_1,n_2)\in N\times N$.  In addition, if $p\in R$, then we let $p_1, \ldots, p_k$ denote the preimages of $p$ under the (trivial) covering map $M^P \rightarrow R$.  \par
Let $h\in C_c(\mathcal{G})$ (i.e., a continuous function with compact support on $\mathcal{G}$) and define a map:
$$h \mapsto (f_h, g_h) \in C_0(M^Q) \oplus C_0(R,M_k)$$
as follows.  For $q\in M^Q$,  
$$f_h(q):=h(q,q)$$ and, for $p\in R-\Sigma_R$, we define
$$g_h(p):=[ h(p_i,p_j)]_{i=1,\ldots,k,j=1,\ldots,k}$$      
where we have used the definition of $\{p_i\}_{i=1}^k$ discussed in the first paragraph of the proof. Finally, for $\tilde{p}\in \Sigma_R$, we define, $g_h(\tilde{p})$ to be the diagonal $k$ by $k$ matrix with entries along the diagonal given by 
$$ h(p_1,p_1), \ldots, h(p_k,p_k) $$
It is now left to the reader to show that this map is well-defined and extends to an isomorphism from $C^*(\mathcal{G})$ to the $C^*$-algebra in the statement of the proposition (i.e., Equation \ref{PropAboutGroupoid}).
\end{proof}
\begin{cor}
The $C^*$-algebra from Definition \ref{groupoidC*algebraGen} (i.e., $C^*(\mathcal{G})$) fits into the following exact sequence: 
\begin{equation}
0\rightarrow C_0(R-\Sigma_R) \otimes M_k \rightarrow C^*(\mathcal{G}) \rightarrow C_0(M^Q) \rightarrow 0
\label{genExaForGro}
\end{equation} 
\end{cor}
Suppose that we have a Riemannian metric on $N$ and let $L^2(N)$ denote the Hilbert space of $L^2$-sections.  If the metric respects the decomposition of $N$, then $L^2(N)$ has the structure of a $\zkz$-Hilbert space (see \cite[Definition 3.2]{Hig}).  That is, we have isometries, $e_i: L^2(R) \rightarrow L^2(N)$, where $i=1,\ldots,k$.  \par     
Using this data and Proposition \ref{PropAboutGroupoid}, there is natural representation (denoted $\rho_{\mathcal{G}}$) of  $C^*(\mathcal{G})$ on $L^2(N)$ defined via
\begin{equation} (\rho_{\mathcal{G}}(f,g) \cdot \xi)(n):= \left\{ \begin{array}{ccc} f(n)\xi(n) & : & n\in M^Q \\ (\sum_{i=1}^k \sum_{j=1}^k e_i M_{g_{ij}} e_j^* \xi)(n) & : & n\in M^P \end{array}\right. 
\label{repOfGro}
\end{equation}
where $M_{g_{ij}}$ denotes the multiplication operator (associated to $g_{ij}$) on $L^2(R)$.  The reader will note that $M^Q \cap M^P=\Sigma \ne \emptyset$.  However, the condition in the equation in Proposition \ref{PropAboutGroupoid} implies that the two possible definitions for $n\in \Sigma$ agree.  We leave the proof that $\rho_{\mathcal{G}}$ is a representation to the reader; who should notice the relationship between matrix multiplication and the interaction of $e_i$ and $e_j^*$ in the formula of $\rho_{\mathcal{G}}$.  
\begin{ex} \label{RosGroEx}
The prototypical example of a groupoid of the form discussed in Definition \ref{groupoidC*algebraGen} is the one constructed in \cite{Ros1}.  The construction is as follows.  Let $(Q,P)$ be a $\zkz$-manifold with diffeomorphism $\phi : \partial Q \rightarrow P \times \zkz$.  We denote by $N$ the manifold without boundary given by $Q\cup_{\partial Q} \partial Q \times [0,\infty)$, where we are identifying $\partial Q$ with $\partial Q \times \{0\}$.  The reader should note that $N$ is usually not compact and that (in the notation of Definition \ref{groupoidC*algebraGen}) $M=\partial Q \times [0,\infty)$ and $\bar{P}= P \times [0,\infty)$. \par    
Let $\mathcal{G}$ denote the groupoid constructed using the process discussed above (see Definition \ref{groupoidC*algebraGen}) and $C^*(Q,P;\zkz)$ denote $C^*(\mathcal{G})$.  The content of Proposition \ref{PropAboutGroupoid} in this case takes the form:
\begin{eqnarray}
C^*(Q,P;\zkz) & \cong & \{(f,g)\in C(Q) \oplus C_0(P\times [0,\infty),M_k) \mid \label{cStarZkz}  \\
& &   g|_{P\times \{0\}}\hbox{ is diagonal and }f|_{\partial Q}=g|_{P\times  \{0\}} \}  \nonumber
\end{eqnarray}
Recall that the statement $f|_{\partial Q}=g|_{P\times  \{0\}}$ in Equation \ref{cStarZkz} is more correctly written as 
$$\alpha(f|_{\partial Q})=g|_{P\times  \{0\}}$$
where $\alpha: C(\partial Q) \cong \oplus_{i=1}^k C(P) \rightarrow M_k(C(P))$ is the diagonal inclusion. \par 
In addition, we have the following exact sequence (see Equation \ref{genExaForGro} or \cite{Ros1}).
\begin{equation*}
0\rightarrow C_0(\field{R})\otimes C(P) \otimes M_k \rightarrow C^*(Q,P;\zkz) \rightarrow C(Q) \rightarrow 0
\end{equation*}
\end{ex}
\begin{remark} \label{emptySetZkz}
If we are given a $\zkz$-manifold of the form $(M,\emptyset)$ (where $M$ is a compact manifold), then we let
$$C^*(M,\emptyset;\zkz):= C(M)\otimes C^*(pt;\zkz)$$
\end{remark}
\begin{ex} \label{GroZkzWithBound}
In this example, we consider the case of a $\zkz$-manifold with boundary.  We will form two $C^*$-algebra; they are analogous to $C_0({\rm int}(W))$ and $C(W)$ in the case of a compact manifold with boundary $W$.  To fix notation, let $(\bar{Q},\bar{P})$ be a $\zkz$-manifold with boundary, $(Q,P)$.  \par
We begin with the $C^*$-algebra which is analogous to $C_0({\rm int}(W))$.  In the notation of our basic setup (see the discussion preceding Definition \ref{groupoidC*algebraGen}), let
$$ M^Q= \bar{Q}-Q, M^P= \partial \bar{Q}-Q \times [0,\infty), \Sigma= \partial \bar{Q}-Q, R={\rm int}(\bar{P})\times [0,\infty), \Sigma_R={\rm int}(\bar{P}) $$   
and form the associated groupoid $C^*$-algebra, which will be denoted by $C^*_0(\bar{Q},\bar{P};\zkz)$.  Proposition \ref{PropAboutGroupoid} implies that 
\begin{eqnarray*}
C^*_0(\bar{Q},\bar{P},\tilde{\pi}) & = & \{(f,g)\in C(\bar{Q}) \oplus C_0({\rm int}(\bar{P})\times [0,\infty),M_k) \\
& & |\  f|_Q =0 \ , \ g|_{\bar{P}\times \{0\}}\hbox{ is diagonal and }f|_{\partial \bar{Q} - Q}=g|_{{\rm int}(\bar{P})\times  \{0\}} \} 
\end{eqnarray*}
Again the identification discussed in Proposition \ref{PropAboutGroupoid} has been used here. \par
Next, we discuss the $C^*$-algebra which is analogous to $C(W)$ in the case of a manifold with boundary.  Again, using the setup discussed before Definition \ref{groupoidC*algebraGen}, let
$$ M^Q= \bar{Q}-{\rm int}(Q), M^P= \partial \bar{Q} - {\rm int}(Q) \times [0,\infty), \Sigma= \partial \bar{Q}-{\rm int}(Q), R=\bar{P}\times [0,\infty), \Sigma_R=\bar{P} $$
In this case, we denote the associated groupoid $C^*$-algebra by $C^*(\bar{Q},\bar{P};\zkz)$.  Proposition \ref{PropAboutGroupoid} implies that
\begin{eqnarray*}
C^*(\bar{Q},\bar{P},\tilde{\pi}) & = & \{(f,g)\in C(\bar{Q}) \oplus C_0(\bar{P}\times [0,\infty),M_k) \\
& & |\  g|_{\bar{P}\times \{0\}}\hbox{ is diagonal and }f|_{\partial \bar{Q}-{\rm {\rm int}}(Q)}=g|_{\bar{P}\times  \{0\}} \}
\end{eqnarray*}
Note that we have (yet again) used the identification discussed in Proposition \ref{PropAboutGroupoid}. \par
Moreover, we have the following exact sequence:
\begin{equation} \label{zkzShortExact}
0 \rightarrow C^*_0(\bar{Q},\bar{P};\zkz) \rightarrow C^*(\bar{Q},\bar{P};\zkz) \rightarrow C^*(Q,P;\zkz) \rightarrow 0 
\end{equation}
This is the $\zkz$-version of the exact sequence:
$$0 \rightarrow C_0({\rm int}(W)) \rightarrow C(W) \rightarrow C(\partial W) \rightarrow 0$$
in the case of a manifold with boundary $W$.
\end{ex}
We now consider natural classes in the $K$-homology of these groupoid $C^*$-algebras.  We have followed \cite{Ros1} for this development.  The setup is as follows.  Let $(Q,P)$ be a ${\rm spin^c}$ $\zkz$-manifold with ${\rm dim}(Q)=n$ and $D$ be the Dirac operator on it (possibly twisted by a $\zkz$-vector bundle).  Let $S_Q$ denote the Dirac bundle to which $D$ is associated.  Then $S_Q$ extends to a Dirac bundle on $N$ (which we denote by $S_N$) and $D$ also extends to all of $N$.  Let the extension of the operator, $D$, to $N$ be denoted by $D_N$.  By Theorem 10.6.5 of \cite{HR}, we can form $[D_N]\in K^{-n}(C_0(N))$.  Moreover, this class is equivariant with respect to the groupoid associated to $(Q,P)$; hence $D_N$ defines a class in $K^{-n}(C^*(Q,P;\zkz))$.  To simplify notation, the class produced from this construction will be denoted by $[D]\in K^{-n}(C^*(Q,P;\zkz))$.  This class is represented by the following Fredholm module:
$$(L^2(N),\rho,\chi(D))$$
where 
\begin{enumerate}
\item $L^2(N)$ denotes the completion of the compactly supported, smooth sections of the spinor bundle (possibly twisted by a $\zkz$-vector bundle; see \cite[Section 10.1]{HR} for details);
\item $\rho_{(Q,P)}: C^*(Q,P) \rightarrow L^2(N)$ is defined in Equation \ref{repOfGro};
\item $\chi$ is a normalizing function (see \cite[Definition 10.6.1]{HR});
\item $D$ is the Dirac operator on $N$ (possibly twisted by a $\zkz$-vector bundle); 
\end{enumerate} 
\par
Moreover, the same construction applies verbatim to produce a class $[D] \in K^{-n}(C_0^*(\bar{Q},\bar{P};\zkz))$ where $(\bar{Q},\bar{P})$ is a ${\rm spin^c}$ $\zkz$-manifold with boundary and $D$ is again the Dirac operator (possibly twisted by a $\zkz$-vector bundle).  The next proposition summarizes basic properties of the $K$-homology classes produced by this construction.
\begin{prop} \label{propZkzClass}
Let $(Q,P)$ be a ${\rm spin^c}$ $\zkz$-manifold with $dim(Q)=n$ and $D_{(Q,P)}$ denote the Dirac operator on $(Q,P)$ (possibly twisted by a $\zkz$-vector bundle).  Then the associated class $[D_{(Q,P)}]\in K^{-n}(C^*(Q,P;\zkz))$ has the following properties:
\begin{enumerate} 
\item Let $(W,V)$ be a ${\rm spin^c}$ $\zkz$-vector bundle over $(Q,P)$ with the dimension of the fibers equal to $2k$, and let $(E,F)$ denote a vector bundle over $(Q,P)$.  Moreover, assume that $[D_{(Q,P)}]$ is the Dirac operator of $(Q,P)$ twisted $(E,F)$.  We denote the ${\rm spin^c}$ $\zkz$-manifold produced by the vector bundle modification of $(Q,P)$ by $(W,V)$ by $(Q^W,P^V)$.  Let $\pi$ denote the projection $(Q^W,P^V) \rightarrow (Q,P)$ and $\tilde{\pi}$ denote the induced map from $C^*(Q,P;\zkz)$ to $C^*(Q^W,P^V;\zkz)$.  Finally, let $[D_{(Q^W,P^V)}] \in K^{-p-2k}(C^*(Q^W,P^V;\zkz))$ denote the class associated to the Dirac operator (now twisted by vector bundle $(H_Q\otimes \pi^*(E), H_P\otimes \pi^*(F))$; see Definition \ref{vbmzkz} and \cite{BHS}).  Then 
$$\tilde{\pi}^*([D_{(Q^W,P^V)}]=[D_{(Q,P)}] \in K^{-n}(C^*(Q,P;\zkz)$$
\item Suppose that $(Q,P)$ is the boundary of the $\zkz$-manifold, $(\bar{Q},\bar{P})$.  If $\partial$ is 
the boundary map of the six-term exact sequence in K-homology associated to the short exact sequence
$$0 \rightarrow C^*_0(\bar{Q},\bar{P};\zkz) \rightarrow C^*(\bar{Q},\bar{P};\zkz) \rightarrow C^*(Q,P;\zkz) \rightarrow 0$$
then 
$$\partial [D_{\bar{Q},\bar{P}}] = [D_{(Q,P)}] \in K^*(C^*(Q,P;\zkz))$$  
\end{enumerate}
\end{prop}
\begin{proof}
Our proof of the first statement is modelled on the proof of Proposition 3.6 in \cite{BHS}.  A partition of unity argument reduces the proof to the case when $(W,V)$ is a trivial $\zkz$-vector bundle.  We may assume that $Q$ is connected; hence $(Q^W,P^V)=(Q\times S^{2n},P\times S^{2n})$ where $2n$ is the fiber dimension of the bundle $(W,V)$.  The Dirac operator on $C^*(Q^W,P^V)$ is given by $D_{Q,P}\otimes I + I\otimes D_{S^{2n}}$.  As in the commutative case, the Hilbert space $L^2(N\times S^{2n})$ decomposes as follows:
$$L^2(N\times S^{2n})=L^2(N)\otimes L^2(S^{2n}) \cong L^2(N)\otimes ker(D_{S^{2n}}) \oplus L^2(N)\otimes ker(D_{S^{2n}})^{\bot} $$
The Fredholm module $[D_{Q\times S^{2n},P\times S^{2n}}]$ respects this decomposition.  Moreover, using the fact that $ker(D_{S^{2n}})$ is one dimensional and specific form of the Dirac operator, the restriction to the first factor is equal to the module $[D_{Q,P}]$.  \par
To complete the proof, we must show that the restriction to the second factor is trivial in K-homology.  Let $T$ be the partial isometry part of the polar decomposition of $D_{S^{2n}}$ and $\gamma$ be the grading operator on $L^2(N)$.  It is an exercise to show that $\gamma \otimes T$ is an odd-graded involution which anticommutes with operator in the Fredholm module and commutes with the action of $C^*(Q,P)$.  Lemma 2.7 of \cite{BHS} then implies that the restriction to the second factor is a trivial Fredholm module over $C^*(Q,P)$; this completes the proof of the first statement.  
\par
To prove the second statement of the theorem, we use the following commutative diagram (see \cite[Section 9.6]{HR} for details; in particular for the definitions of the relevant $C^*$-algebras). 
\begin{center}
$\minCDarrowwidth15pt\begin{CD}
0 @>>> S(C^*(Q,P)) @>>> C(C^*(Q,P)) @>>> C^*(Q,P) @>>> 0  \\
@. @VV\alpha V @VVV @|  \\
0 @>>> C(C^*(\bar{Q},\bar{P}),C^*(Q,P)) @>>> Q(C^*(\bar{Q},\bar{P}),C^*(Q,P)) @>>> C^*(Q,P) @>>> 0 \\
@. @AA\beta A @AAA @| \\
0 @>>> C^*_0(\bar{Q},\bar{P}) @>>> C^*(\bar{Q},\bar{P}) @>>> C^*(Q,P) @>>> 0
\end{CD}$
\end{center}
Let $b$ denote the boundary map associated to the first exact sequence and $[d]$ denote the Dirac operator on $\field{R}$.  Standard results (see \cite[Section 9.6]{HR}) imply that $\beta^*$ is an isomorphism, $\partial= b \circ \alpha^* \circ (\beta^*)^{-1}$ and $b([d]\times [D_{Q,P}])=[D_{Q,P}]$.  This reduces the proof to showing  
$\alpha^* \circ (\beta^*)^{-1}([D_{\bar{Q},\bar{P}}])=[d]\times [D_{Q,P}]$.  \par
To this end, the reader (upon recalling the notation from Example \ref{GroZkzWithBound}) can verify that 
$$C(C^*(\bar{Q},\bar{P}),C^*(Q,P)))\cong C_0^*(\bar{N}\cup_{\partial \bar{N}} \partial \bar{N} \times [0,\infty), \bar{P} \times [0,\infty)\times [0,\infty))$$  
As such, $C(C^*(\bar{Q},\bar{P}),C^*(Q,P))$ has an associated Dirac class which we denote by $(\mathcal{H},\hat{\rho},F)$.   
The proof will be complete upon showing that
\begin{equation*}
\beta^*(\mathcal{H},\hat{\rho},F) \sim [D_{\bar{Q},\bar{P}}] \hbox{ and }
\alpha^*(\mathcal{H},\hat{\rho},F) \sim [d]\times[D_{Q,P}]
\end{equation*}
For the first of these statements, consider
$$ \beta^*(\mathcal{H},\hat{\rho},F)=[\mathcal{H},\hat{\rho}\circ \beta, F)]\sim [p\mathcal{H},\hat{\rho}\circ \beta,pFp]$$
where $p$ is the projection onto the image of $(\hat{\rho} \circ \beta)$ (i.e., $C^*_0(\bar{Q},\bar{P})$).  By construction, $p\mathcal{H}\cong L^2(\bar{N})$ and $\hat{\rho}\circ \beta=\rho_{(\bar{Q},\bar{P})}$ on $L^2(\bar{N})$.  Moreover, for each $(f,g)\in C_0^*(\bar{Q},\bar{P})$,  
$$\rho_{(\bar{Q},\bar{P})}(f,g)(pFp\chi(D_{\bar{Q},\bar{P}})+\chi(D_{\bar{Q},\bar{P}})pFp)\rho_{\bar{Q},\bar{P}}(f,g)^*\ge 0 \hbox{ mod } \mathcal{K}(L^2(\bar{N}))$$ 
This follows since (by \cite[Section 10.8]{HR}) it holds for $(f,g)$ in the image of the inclusion of $C_0({\rm int}(N))$ into $C^*(\bar{Q},\bar{P})$ and there exists an approximate unit (for all of $C_0^*(\bar{Q},\bar{P})$) in this image.  Finally, \cite[Proposition 8.3.16]{HR} implies $\beta^*(\mathcal{H},\hat{\rho},F) \sim [D_{\bar{Q},\bar{P}}]$.  The proof that $\alpha^*(\mathcal{H},\hat{\rho},F) \sim [d]\times[D_{Q,P}]$ is similar; one replaces the projection, $p$, above with the projection onto $L^2(N\times (0,\infty))$.  This completes the proof that $\partial [D_{\bar{Q},\bar{P}}]=[D_{Q,P}]$. 
\end{proof}
Following \cite{Ros1}, we replace the collapse to point map in classical index theory with the $*$-homomorphism
\begin{eqnarray}
c: & C^*(pt;\zkz) & \rightarrow C^*(Q,P;\zkz)  \label{cMapZkz}\\
& h & \mapsto h(0) 1_Q \oplus (1_P \otimes h) \nonumber
\end{eqnarray} 
where we have used Equation \ref{cStarZkz} and the fact that
$$C^*(pt;\zkz) \hookrightarrow \{ f\in C_0([0,\infty), M_k)| f(0) \hbox{ diagonal }\}$$
\begin{define} \label{anaIndAndHomc}
Using the notation of the previous paragraph, we define the analytic index of $D$ to be $c^*([D])$.  
\end{define}

\subsection{Map between geometric and analytic $K$-homology with coefficients} 
\begin{define} \label{ctsToStarHom}
Given a continuous map, $f: (Q,P) \rightarrow X$, we define a $*$-homomorphism via 
\begin{eqnarray*}
\tilde{f}: & C(X) \otimes C^*(pt;\zkz) & \rightarrow C^*(Q,P;\zkz) \\
& g\otimes h & \mapsto h(0) f_Q^*(g) \oplus (f_P^*(g) \otimes h)
\end{eqnarray*} 
where we have denoted by $f_Q^*$ and $f_P^*$ the $*$-homomorphism induced by $f|_Q$ and $f|_P$.  Note that we have also used Equation \ref{cStarZkz} and the fact that
$$C^*(pt;\zkz) \hookrightarrow \{ g\in C_0([0,\infty), M_k) \: | \: g(0) \hbox{ diagonal }\}$$
Let $\tilde{f}^*$ denote the map induced on $K$-homology by this $*$-homomorphism. 
\end{define}
In the case when $X=pt$, the $*$-homomorphism from Definition \ref{ctsToStarHom}, $\field{C}\otimes C^*(pt;\zkz) \rightarrow C^*(Q,P;\zkz)$, is the same as the one defined in Equation \ref{cMapZkz}. (One must first identify $\field{C} \otimes C^*(pt;\zkz)$ with $C^*(pt;\zkz)$).
\begin{define}
Let $X$ be a compact Hausdorff space.  Let $\Phi$ be the map between geometric $\zkz$-cycles and analytic $\zkz$-cycles defined by
\begin{equation*}
((Q,P),(E,F),f) \mapsto \tilde{f}^*([D_{(E,F)}])
\end{equation*}
where $[D_{(E,F)}] \in K^*(C^*(Q,P;\zkz))$ is the class of the Dirac operator twisted by $(E,F)$ and $\tilde{f}^*$ is the map on $K$-homology induced from $f$ (see Definition \ref{ctsToStarHom}). \label{defIsoGeoAnazkz} 
\end{define} 
Our first goal is to show that the map, $\Phi$, is well-defined (i.e., show that the class in analytic $K$-homology is invariant under the relations on the geometric $\zkz$-cycles).  The proof for the disjoint union operation is trivial.  \par
The case of vector bundle modification follows from Item 1 in Proposition \ref{propZkzClass}. 
\begin{cor}
Let $X$ be a compact Hausdorff space, $((Q,P),(E,F),f)$ a $\zkz$-cycle over $X$, and $\pi:(W,V) \rightarrow (Q,P)$ a ${\rm spin^c}$ $\zkz$-vector bundle with even dimensional fiber.  Then
$\tilde{f}^*([D_{(E,F)}])= (\tilde{f} \circ \tilde{\pi})^*([D_{(E^W,F^V)}]) $.
\label{invFMindexVBM}
\end{cor}
\begin{proof}
Using Proposition \ref{propZkzClass} and the fact that $\tilde{\pi} \circ \tilde{f}_{(Q,P)}  = \tilde{f}_{(Q^W,P^V)}$ we obtain
\begin{eqnarray*}
\tilde{f}_{(Q^W,P^V)}^*([D_{(Q^W,P^V)}]) & = & (\tilde{\pi} \circ \tilde{f}_{(Q,P)})^*([D_{(Q^W,P^V)}]) \\
& = & \tilde{f}_{(Q,P)}^* ( \tilde{\pi}^*([D_{(Q^W,P^V)}])) \\
& = & \tilde{f}_{(Q,P)}^*([D_{(Q,P)}]) 
\end{eqnarray*}
\end{proof}
This completes the proof that $\Phi$ respects vector bundle modification.  The case of $\zkz$-bordism is less clear.  Our proof is similar to the proof in the commutative case given in \cite{BDT} (also see \cite[Exercise 11.8.10]{HR}). To begin, note that if $(\bar{Q},\bar{P})$ is a $\zkz$-manifold with boundary and $\bar{f}:(\bar{Q},\bar{P})\rightarrow X$ is a continuous map then we can define a $*$-homomorphism via
\begin{eqnarray*}
\widetilde{\bar{f}}:  C(X)\otimes C^*(pt;\zkz) & \rightarrow & C^*(\bar{Q},\bar{P},\tilde{\pi}) \\
  h\otimes r & \mapsto & (h \circ f) \oplus ((h \circ f|_{\bar{P}}) \otimes r)
\end{eqnarray*}
The reader can verify the next lemma.
\begin{lemma} \label{factorThrough}
Let $X$ be a compact Hausdorff space and $((Q,P),(E,F),f)$ be a $\zkz$-cycle over $X$, which is the boundary of $((\bar{Q},\bar{P}),(\bar{E},\bar{F}),\bar{f})$.  Then the map $\tilde{f}^*$ (see Definition \ref{ctsToStarHom}) factors through the map induced from the inclusion of the boundary.  
\end{lemma}
\begin{theorem}
If $((Q,P),(E,F),f)$ is a $\zkz$-cycle which is a boundary, then $\tilde{f}^*([D_{(Q,P),(E,F)}])=0$ in $K^*_{ana}(C(X);\zkz)(=K^*(C(X)\otimes C^*(pt;\zkz)))$. \label{coborInvClass}
\end{theorem}
\begin{proof}
Denote by $((\bar{Q},\bar{P}),(\bar{E},\bar{F}),\bar{f})$ the cycle which has boundary, $((Q,P),(E,F),f)$.  
Applying the analytic K-homology functor to the short exact sequence in Equation \ref{zkzShortExact} (see the discussion following Example \ref{GroZkzWithBound}), we get long exact sequence:
$$ \rightarrow K^*(C_0^*(\bar{Q},\bar{P})) \stackrel{\partial}{\rightarrow} K^{*+1}(C^*(Q,P)) \stackrel{r^*}{\rightarrow} K^{*+1}(C^*(\bar{Q},\bar{P})) \stackrel{i^*}{\rightarrow} K^{*+1}(C_0^*(\bar{Q},\bar{P})) \rightarrow $$
Item 2) of Proposition \ref{propZkzClass} and exactness imply that
$$0=(r^* \circ \partial)([D_{(\bar{Q},\bar{P}),(\bar{E},\bar{F})}])=r^*([D_{(Q,P),(E,F)}])$$
Moreover, using Lemma \ref{factorThrough}, we have that 
$$\tilde{f}^*([D_{(Q,P),(E,F)}])=\tilde{\bar{f}}^*(r^*([D_{(Q,P),(E,F)}])=0$$
\end{proof}
\begin{theorem}
Let $X$ be a finite CW-complex.  Then the map, $\Phi$, constructed in Definition \ref{defIsoGeoAnazkz} is an isomorphism. \label{thmIsoGeoAnazkz}
\end{theorem}
\begin{proof}
To begin, we prove that the Bockstein sequences of the analytic and geometric models fit into the following commutative diagram (we show only the $K_0$-part of the relevant commutative diagram): 
\begin{center}
$\begin{CD}
@>>> K^{geo}_0(X) @>>> K^{geo}_0(X) @>>> K^{geo}_0(X;\zkz) @>>>  \\
@. @VV\mu V @VV\mu V @VV\Phi V  \\
@>>> K^{ana}_0(X) @>>> K^{ana}_0(X) @>>> K^{ana}_0(X;\zkz) @>>> 
\end{CD}$
\end{center}
where we recall that $\mu$ denotes the natural transformation defined from geometric K-homology to analytic K-homology (see Equation \ref{isoNoCoeff} in the Introduction).  Commutativity follows from the following three facts: 
\begin{enumerate}
\item $\mu$ is a group homomorphism and hence commutes with multiplication by $k$;
\item The $C^*$-algebra associated to a $\zkz$-manifold of the form $(M,\emptyset)$ is $C(M)\otimes C^*(pt;\zkz)$ and the $*$-homomorphism from $C(X)\otimes C^*(pt;\zkz)$ to $C(M)\otimes C^*(pt;\zkz)$ is $f \otimes id$ (see Remark \ref{emptySetZkz} for more details).
\item The commutative diagram:
\begin{center}
$\begin{CD}
C(X)\otimes C_0(0,\infty) \otimes M_k @>>> C(X)\otimes C(pt;\zkz)   \\
 @VV\tilde{f}|_P \otimes id \otimes id V   @VV\tilde{f} V  \\
C(P)\otimes C_0(0,\infty) \otimes M_k @>>> C^*(Q,P;\zkz)  
\end{CD}$
\end{center} 
and the fact that $[D_{(Q,P)}]$ is mapped to $[D_P]$ under the map:
$$K^*(C^*(Q,P;\zkz)) \rightarrow K^*(C(P)\otimes C_0(0,\infty) \otimes M_k) \cong K^{*+1}(C(P))$$
\end{enumerate}  
The result then follows using the exactness of the Bockstein sequences, the fact that the map between geometric and analytic $K$-homology is an isomorphism, and the Five Lemma.
\end{proof}
\section{Connection with the Freed-Melrose Index Theorem}
The isomorphism from geometric $K$-homology (i.e., the Baum-Douglas model) to analytic $K$-homology (i.e., Kasparov model via Fredholm modules) leads to the Atiyah-Singer index theorem by considering the case of a point.  That is, from the commutativity of the following diagram: 
\begin{center}
$$\xymatrix{
 K^{geo}_0(pt) \ar[rdd]_{ind_{top}}  
 \ar[rr]^{\mu} 
 & & K^{ana}_0(pt)  \ar[ldd]^{ind_{ana}}  
 \\ \\
& \field{Z} &
} $$
\end{center}
In the case of geometric $K$-homology with coefficients in $\zkz$, we have an analogous diagram; namely,
\begin{center}
$$\xymatrix{
 K^{geo}_0(pt;\zkz) \ar[rdd]_{ind^{FM}_{top}}  
 \ar[rr]^{\Phi} 
 & & K^{ana}_0(pt;\zkz)  \ar[ldd]^{ind^{FM}_{ana}}  
 \\ \\
& \zkz &
} $$
\end{center}
In this case, the commutativity of the diagram is the statement of the Freed-Melrose index theorem (c.f., \cite{Ros1}). \par
We can also consider pairings.  For example, the pairing 
$$K^0(X) \times K_0(X;\zkz) \rightarrow K_0(pt;\zkz)$$
is given (analytically) by the Kasparov product:
\begin{equation}
KK^0(\field{C},C(X)) \times KK^0(C(X)\otimes C^*(pt;\zkz),\field{C}) \rightarrow KK^0(C^*(pt;\zkz),\field{C}) \label{KasParK0}
\end{equation}
While, on the geometric side, we have the following.  Let $V$ be a vector bunlde over $X$ and $((Q,P),(E,F),f)$ be a geometric $\zkz$-cycle.  Then the pairing is given by: 
$$ind^{FM}_{top}( D^{(Q,P)}_{(E\otimes f^*(V),F\otimes (f|_P)^*(V))})$$
This pairing naturally extends to K-theory and, moreover, is equal to the Kasparov pairing above (i.e., Equation \ref{KasParK0}).  This follows from the associativity of the Kasparov product and the fact that the pairing between a $\zkz$-vector bundle and the Dirac operator on $(Q,P)$ is given by twisting the operator by the bundle.  Finally, using the orginal formulation of the Freed-Melrose index theorem (c.f., Corollary 5.4 in \cite{FM}), we have that this pairing is also equal to the Fredholm index of the operator $D^{Q}_{(E\otimes f^*(V))}$ with the Atiyah-Patodi-Singer boundary conditions (see \cite{APS1}).  \par
We have a similar pairing between the groups $K^1(X)$ and $K_1(X;\zkz)$.  However, the reader may recall that the pairing between $K^1(X)$ and $K_1(X)$ is given by the index of a Toeplitz operator (see the introdcution of \cite{DZ} for details).  Thus one is led to ask if there is an analogous index theorem for $\zkz$-manifolds.  To the author's knowledge, this is unknown.  However, in \cite{DZ}, an index theorem for Toeplitz operators on odd-dimensional manifolds with boundary is developed.  Thus it is natural to ask if, in the case of a $\zkz$-manifold, the mod k reduction of this index is a topological invariant and, moreover, if it is equal to the pairing between $K^1(X)$ and $K_1(X;\zkz)$.  Both these questions represent future work. \vspace{0.25cm} \\    
{\bf Acknowledgments} \\
I would like to thank my PhD supervisor, Heath Emerson, for useful discussions on the content and style of this document.  In addition, I thank Nigel Higson, Jerry Kaminker, John Phillips, and Ian Putnam, for discussions.  In particular, I would like to thank Jerry Kaminker for bringing the problem of a ``Toeplitz index theorem" for $\zkz$-manifolds and the reference \cite{DZ} to my attention.  This work was supported by NSERC through a PGS-Doctoral award.  

Email address: rjdeeley@uni-math.gwdg.de  \\
{ \footnotesize MATHEMATISCHES INSTITUT, GEORG-AUGUST UNIVERSIT${\rm \ddot{A}}$T, BUNSENSTRASSE 3-5, 37073 G${\rm \ddot{O}}$TTINGEN, GERMANY}
\end{document}